\documentclass[15pt]{article}

\usepackage{amsmath,amsfonts,amssymb,amsthm,amscd}

\newcounter{stepctr}
{\end{list}}
\newtheorem{thm}{Theorem}[section]
\newtheorem{prop}[thm]{Proposition}
\newtheorem{cor}[thm]{Corollary}

\theoremstyle{definition}
\newtheorem{dfn}[thm]{Definition}

\newtheorem{rema}[thm]{Remark}
\newtheorem{lem}[thm]{Lemma}
\newtheorem{prob*}{Open problem}
\newcommand{\demo}{\begin{proof}}

\newcommand{\R}{\ensuremath{\mathcal{R}}}

\newcommand{\N}{\mathbb{N}}

\newcommand{\C}{\mathbb{C}}

\def\ll^2{{\mathcal L}(\ell^2(\N))}

\def\f^0x{{\mathcal F^0}(X) }

\title
{\bf  Generalization of Kato's decomposition}

\author{   Zakariae Aznay, Abdelmalek Ouahab, Hassan Zariouh }

\date{}
\begin{document}

\maketitle \thispagestyle{empty}

\begin{abstract}\noindent\baselineskip=10pt
The  Kato's decomposition \cite[Theorem 4]{kato} is  generalized to semi-B-Fredholm operators.

\end{abstract}

 \baselineskip=15pt
  \baselineskip=15pt

\section{Introduction and preliminaries}
 Let $T \in L(X);$ where   $L(X)$ is the  Banach algebra of bounded linear operators acting on an infinite dimensional complex Banach space $X.$  We denote by $T^{*}$    the dual of $T,$ by $\alpha(T)$ the dimension of the kernel $\mathcal{N}(T)$ and by  $\beta(T)$    the codimension of the range $\R(T).$    A  subspace $M$ of $X$ is $T$-invariant if $T(M)\subset M$ and in this case  $T_{M}$ means the restriction of $T$ on $M.$  We say that  $T$   is completely reduced   by  a pair    $(M,N)$   ($(M,N) \in Red(T)$ for brevity) if $M$ and  $N$ are  closed $T$-invariant subspaces  of $X$ and  $X=M\oplus N;$   here  $M\oplus N$ means that $M\cap N=\{0\}.$    Let  $n\in\N,$  we denote by   $T_{[n]}:=T_{\mathcal{R}(T^{n})}$  (in particular, $T_{[0]} = T$). It is clear that  $(\alpha(T_{[n]}))_{n}$ and $(\beta(T_{[n]}))_{n}$ are  decreasing sequences. We call      $m_T:=\mbox{inf}\{n \in \N  :  \mbox{inf}\{\alpha(T_{[n]}),\beta(T_{[n]})\}<\infty\}$  (with $\mbox{inf}\emptyset=\infty$) the essential  degree of   $T.$ Following \cite{mbekhta} we say that $T$ has   finite essential ascent (resp., descent)  if  $a_{e}(T):=\mbox{inf}\{n \in \N  :  \alpha(T_{[n]})<\infty\}<\infty$  (resp., $d_{e}(T):=\mbox{inf}\{n \in \N  :  \beta(T_{[n]})<\infty\}<\infty$).   Note that $m_{T}= \mbox{inf}\{a_{e}(T),d_{e}(T)\}.$ Moreover,  if $T$ has finite essential ascent  and  essential descent  then $m_{T}=a_{e}(T)=d_{e}(T).$    Operators with finite essential ascent or descent seem to have been first studied in \cite{grabiner1},  these operators played an important role  in many papers, see for example \cite{berkani, berkani-sarih,grabiner, mbekhta}.

 \par  It is easily seen that    the definition given in  \cite{berkani,berkani-sarih} of  an   \emph{upper semi-B-Fredholm}  (resp., \emph{lower  semi-B-Fredholm}) operator $T$  is equivalent  to the following definition given in  \cite{mbekhta}: $a_{e}(T)<\infty$  and   $\mathcal{R}(T^{a_{e}(T)+1})$ is closed (resp., $d_{e}(T)<\infty$  and $\mathcal{R}(T^{d_{e}(T)})$  is closed).
  $T$ is called  \emph{semi-B-Fredholm} (resp., \emph{B-Fredholm}) if  it  is an upper or (resp., and) a lower semi-B-Fredholm. According to Corollary \ref{cor1} below, if $T$ is a semi-B-Fredholm   then $\R(T^{m_T})$ is closed and $m_T=\mbox{min}\{n \in \N: \R(T^n)  \text{ is closed and } \mbox{inf}\{\alpha(T_{[n]}),\beta(T_{[n]})\}<\infty\},$  and in this case   the    index  of $T$ is  defined  by $\mbox{ind}(T) := \alpha(T_{[m_{T}]})-\beta(T_{[m_{T}]}).$  If $T$ is an upper semi-B-Fredholm (resp.,  lower semi-B-Fredholm,    semi-B-Fredholm, B-Fredholm) with essential degree $m_{T}=0,$ then $T$ is said to be  an upper semi-Fredholm (resp.,  lower semi-Fredholm,   semi-Fredholm, Fredholm) operator.
\par  The degree of stable iteration $d:=\mbox{dis}(T)$ of $T$ is defined as  $d=\mbox{inf}\,\Delta(T);$   where    $$\Delta(T):=\{m\in\N \,:\,  \alpha(T_{[m]})= \alpha(T_{[r]}),\,\forall r \in \N  \,  \, r\geq m \}.$$  We say that $T$ is semi-regular if $\R(T)$ is closed and $d=0,$ and we say that    $T$ is quasi-Fredholm if there exists $n \in \N$ such that  $\R(T^{n})$  is closed  and $T_{[n]}$  is semi-regular.  From \cite[Lemma 12]{mbekhta}, it follows that if $T$ is a quasi-Fredholm operator then  $d=\mbox{min}\{n \in \N :  \R(T^{n}) \text{ is closed and } T_{[n]} \text{ is semi-regular}\}.$  For more details about these definitions, see for examples \cite{mbekhta,labrousse}. Note that every semi-B-Fredholm operator is quasi-Fredholm \cite[Proposition 2.5]{berkani-sarih}.  We recall \cite[p. 288]{kato} that  $v(T):=\mbox{inf}\{n \in \N : \mathcal{N}(T)\not\subset \R(T^{n})\}$ which play a crucial role in several studies of this direction  \cite{kato, west}. Hereafter we call $v:=v(T)$  the degree of the  semi-regularity of $T.$
\vspace{0.2cm}
 \par In \cite{west1,west},  West simplified  the proof of a special case of  a Kato's result (named Kato's decomposition) \cite[Theorem 4]{kato}  which establishes that  if   $T$ is a semi-Fredholm operator  which is not semi-regular then there exists $(M,N) \in Red(T)$   such that $\mbox{dim}\,N<\infty,$ $T_{M}$ is semi-regular (i.e. $\mbox{dis}(T_{M})=m_{T}=0$) and $T_{N}$ is nilpotent of degree $v.$ To prove this result \cite{west},  West used the notion of  jumping $"\mbox{jump}(T)"$ of a  semi-Fredholm operator $T$  introduced by himself, and the degree $v$  of  the   semi-regularity of $T.$  In this paper we generalize the concept of jumping  to semi-B-Fredholm operators. By using our  new concept of jumping,    the degree of stable iteration $d$  and the essential degree $m_{T}$ of an operator $T,$   we generalize the Kato's decomposition  to    semi-B-Fredholm operators. In other words, we prove in Theorem \ref{thmm} that if $T$ is semi-B-Fredholm then  there exists $(M,N) \in Red(T)$   such that $\mbox{dim}\,N<\infty,$ $(T_{M})_{[m_{T}]}$ is semi-regular   (i.e. $\mbox{dis}(T_{M})=m_{T}$)  and $T_{N}$ is nilpotent of degree $d.$ Among other things,  some   additional results for   semi-B-Fredholm operators are given, and  some known results in Kato's theory are generalized (Theorem \ref{thm0},  Corollary  \ref{corrr} and Theorem \ref{thmunique}).

\section{Additional results for  semi-B-Fredholm operators}
The following proposition  extend \cite[Lemma 11]{mbekhta}.
\begin{prop} Let $T \in L(X)$ and  let $n,m \in \N.$ If  $\R(T^{n})$ is closed and $\R(T^{m})+\mathcal{N}(T^{n})$ is closed then  $\R(T^{m+n})$ is closed.
\end{prop}
\begin{proof} Let $(x_{k})_{k} \subset X$ be a sequence  such that    $T^{m+n}x_{k}$ $\longrightarrow$ $y \in X.$ Since  $\R(T^{n})$ is closed then  there exists $z\in X$ such that $y=T^{n}z$  and  the operator $\overline{T^{n}} \in L(X/\mathcal{N}(T^{n}))$ induced by $T^{n}$ is bounded below.  Hence   $d(T^{m}x_{k}-z, \mathcal{N}(T^{n}))$ $\underset{k \rightarrow \infty}{\longrightarrow}$ $0;$  where $d(a,A)$ is the distance between $a \in X$ and a subspace $A$ of $X.$   As by hypothesis  $\R(T^{m})+\mathcal{N}(T^{n})$ is closed    then  $z\in \R(T^{m})+\mathcal{N}(T^{n})$   and thus  $y \in \R(T^{m}).$  Consequently  $\R(T^{m+n})$ is closed.
 \end{proof}
\begin{prop}\label{prop4}
Let $T \in L(X).$ If there exists  $n\in \N$ such that $\min\{\alpha(T_{[n]}),\beta(T_{[n]})\}<\infty,$ then
$\R(T^n)$ is closed   if and only if   there exists   an integer $m \geq n+1$ such that $\R(T^m)$ is closed. Moreover, in this case  $T$  is a  semi-B-Fredholm.
\end{prop}
\begin{proof}
Suppose that $\R(T^n)$ is closed, as  $\min\{\alpha(T_{[n]}),\beta(T_{[n]})\}<\infty$ then $T_{[n]}$ is semi-Fredholm operator, and so   $T_{[n]}^{m-n}$ is semi-Fredholm  for every integer $m\geq n.$ Thus $\R(T_{[n]}^{m-n})=\R(T^{m})$ is closed in $\R(T^{n}),$ and then   $\R(T^{m})$ is closed. Conversely,  we have $k_n(T) \leq \min\{\alpha(T_{[n]}),\beta(T_{[n]})\}<\infty;$ where   $k_n(T):=\mbox{dim}\frac{\mathcal{N}(T_{[n]})}{\mathcal{N}(T_{[n+1]})}.$ As it is already mentioned that  $(\alpha(T_{[k]}))_{k}$ and $(\beta(T_{[k]}))_{k}$ are  decreasing sequences,  then $k_i(T)<\infty$ for every $i \geq n.$    Let  $m$ be  an integer such that $m \geq n+1$ and $\R(T^{m})$ is closed, then \cite[Lemma 12]{mbekhta}  implies that   $\R(T^n)$ is closed. Thus $T_{[n]}$ is semi-Fredholm.
\end{proof}
For $T \in L(X),$ we denote by  $m_{T}$ the essential degree of $T$   defined     by    $m_T:=\rm{inf}\{n \in \N  :  \rm{inf}\{\alpha(T_{[n]}),\beta(T_{[n]})\}<\infty\},$ with $\rm{inf}\emptyset=\infty.$   Corollary \ref{cor00} below shows also that
$$m_{T}=\left\lbrace
\begin{array}{ll}
a_{e}(T) & \mbox{if $a_{e}(T)<\infty$}\\
d_{e}(T) & \mbox{if $d_{e}(T)<\infty$}\\
\infty &  \mbox{else}
\end{array}
\right.$$
where $a_{e}(T)$ and $d_{e}(T)$ are respectively, the essential ascent and the essential descent of $T.$
\begin{cor}\label{cor1}
If  $T \in L(X)$ is  a  semi-B-Fredholm operator  then $\R(T^{m_T})$ is closed and $m_T=\mbox{min}\{n \in \N: \R(T^n)  \text{ is closed and } T_{[n]}  \text{ is }  \text{ semi-Fredholm}\}.$
\end{cor}
\begin{proof}
 Since $T$ is a semi-B-Fredholm then $\mbox{min}\{\alpha(T_{[m_T]}),\beta(T_{[m_T]})\}<\infty.$   On the other hand, there exists   an integer $n \geq m_T$ such that $\R(T^n)$ is closed. Then by Proposition \ref{prop4},  $\R(T^{m_T})$ is closed and $T_{[m_T]}$   is a  semi-Fredholm. Thus  $m_T=\mbox{min}\{n: \R(T^n)  \text{ is closed and } T_{[n]}  \text{ is }  \text{ semi-Fredholm}\}.$
\end{proof}
According to   the properties  of the index of  semi-B-Fredholm operators   \cite[Proposition 2.1]{berkani-sarih}, we   immediately obtain the next corollary.
\begin{cor}\label{cor0} Let $T \in L(X).$ If there exist $n,m \in \N$ such that $\R(T^{n})$ is closed and   the operator $T_{[n]}$ is an  upper (resp., a lower) semi-Fredholm and $\beta(T_{[m]})<\infty$ (resp., $\alpha(T_{[m]})<\infty$), then $\R(T^{m})$ is closed,   $T_{[m]}$ and $T_{[n]}$  are  Fredholm. Consequently, if $T$ is a B-Fredholm operator then $m_T=\mbox{min}\{n: \R(T^n)  \text{ is closed and } T_{[n]}  \text{ is }  \text{Fredholm}\}=\mbox{min}\{n  :  \mbox{max}\{\alpha(T_{[n]}),\beta(T_{[n]})\}<\infty\}=a_{e}(T)=d_{e}(T).$
\end{cor}
\begin{prop} Let $T \in L(X)$ and  let $n$  be a strictly positive integer.   Then $T$ is a semi-B-Fredholm operator if and only if  $T^n$ is a semi-B-Fredholm operator.  If this is the case then   $\mbox{ind}(T^n)=n .\mbox{ind}(T).$

\end{prop}
\begin{proof}
 From \cite[Proposition 4.2]{berkani-sarih}, it follows that $T$ is  semi-B-Fredholm  if and only if $T^n$ is  semi-B-Fredholm.  Suppose that $T$ is  semi-B-Fredholm then by \cite[Proposition 2.1]{berkani-sarih} $\R(T^{nm})$ is closed, $T_{[nm]}$ is semi-Fredholm and  $\mbox{ind}(T)=\mbox{ind}(T_{[m]})=\mbox{ind}(T_{[nm]});$ where   $m=m_{T}.$  As  $\R((T^{n})^m)=\R(T^{nm})$ is closed  and $(T^{n})_{[m]}=(T_{[nm]})^n$ is  semi-Fredholm, then $\mbox{ind}(T^n)=\mbox{ind}((T^{n})_{[m]})=\mbox{ind}((T_{[nm]})^n)=n.\mbox{ind}(T_{[nm]})=n.\mbox{ind}(T).$
\end{proof}
\begin{prop}\label{prop2}
Let $T\in L(X)$ and let $n \in \N.$\\
 (i) If $\alpha(T_{[n]}) < \infty$ and $\beta(T_{[n+1]}) < \infty$ then $\beta(T_{[n]}) < \infty.$\\
 (ii) If $\alpha(T_{[n+1]}) < \infty$ and $\beta(T_{[n]}) < \infty$ then $\alpha(T_{[n]}) < \infty.$
\end{prop}
\begin{proof}
Since   $\alpha(T_{[n]})=\alpha(T_{[n+1]})+k_{n}(T) \leq \alpha(T_{[n+1]}) + \beta(T_{[n]})$ and  $\beta(T_{[n]})=\beta(T_{[n+1]})+k_{n}(T) \leq \beta(T_{[n+1]}) + \alpha(T_{[n]}),$ then the proof is complete. Remark that for $n=0,$ the assertion (ii) is also an immediate  consequence of Corollary \ref{cor0}.
\end{proof}
As   an immediate consequence of Proposition \ref{prop2}, we find   \cite[Lemma 11, p.204]{mullerbook}.
\begin{cor}\label{cor00} If  $T \in L(X)$ and $\max\{a_{e}(T), d_{e}(T)\}<\infty$ then  $m_{T}=a_{e}(T)=d_{e}(T).$
\end{cor}
We denote by $F(X)$ (resp.,  $SF_+(X),$  $SF_-(X),$ $SBF_+(X),$  $SBF_-(X)$) the class of Fredholm (resp., upper semi-Fredholm, lower semi-Fredholm, upper semi-B-Fredholm, lower semi-B-Fredholm)  operators.
\begin{prop}\label{prop3}
Let $T\in L(X).$   The following  are equivalent.\\
 (i) $T$ is a Fredholm operator;\\
 (ii)  $\alpha(T)<\infty$ and $\beta(T_{[n]}) < \infty$ for some  integer $n;$\\
(iii) $\beta(T)<\infty$ and $\alpha(T_{[n]}) < \infty$   for some  integer $n.$\\
 Consequently, $F(X)=SF_+(X)\cap SBF_-(X)=SF_-(X)\cap SBF_+(X).$
\end{prop}
\begin{proof}
The implications ``(i) $\Longrightarrow$ (ii)" and  ``(i) $\Longrightarrow$ (iii)" are obvious.\\
(ii) $\Longrightarrow$ (i) We denote by  $p=d_{e}(T)$     and let us  prove that $p=0.$ Suppose to the contrary, that's $p \geq 1.$   We take  $S=T_{[p-1]},$ then $S_{[1]}=T_{[p]}$ and  so $\beta(S_{[1]})< \infty.$  On the other hand, since   $\alpha(T)<\infty$ then  $\alpha(S)<\infty.$        It follows from the assertion (i) of  Proposition \ref{prop2} that  $\beta(S) < \infty$   and this is a contradiction.\\
(iii) $\Longrightarrow$ (i)  It is an   immediate consequence of Corollary  \ref{cor0}.
\end{proof}
Let $n, m \in \N.$  Note that the space   $X \times \mathcal{N}(T^n)$ equipped with the norm defined by  $\|(x,y)\|=\|x\|+\|y\|$ for every $(x,y) \in X\times  \mathcal{N}(T^n),$   is a Banach space.   The  map $S_{n,m}$ : $X \times \mathcal{N}(T^n)$ $\longrightarrow$   $X$ defined by $S_{n,m}(x, y) = T^m x+y$  is a bounded linear operator.
\begin{lem}\label{lem4}
For every $T \in L(X),$ the following hold.\\
(i) If there exists $n \in \N$ such that    $\beta(T_{[n]})< \infty$ then $\R(T^m)+\mathcal{N}(T^k)$ is closed, for all  $m \in \N$ and $k \in \N$ such that $k \geq n.$\\
(ii) $\forall n \in \N$ and $\forall m \in \N^*,$ we have $T$ is a Fredholm if and only if $S_{n,m}$ is a Fredholm.
\end{lem}
\begin{proof} (i) Let  $n \in \N$ such that    $\beta(T_{[n]})< \infty$ and   $m\in \N.$
  The case of $m=0$ is trivial. Suppose that $m \neq 0,$ we have  $\R(S_{n,m}) = \R(T^m)+\mathcal{N}(T^n).$    From \cite[Lemma 3.2]{kaashoek} we have  $\beta(S_{n,m})=\displaystyle\sum_{i=n}^{n+m-1}\beta(T_{[i]}).$ As $\beta(T_{[n]})< \infty$ and $(\beta(T_{[k]}))_{k}$ is a  decreasing sequence then $\beta(S_{n,m})\leq m\beta(T_{[n]})< \infty,$ thus    $\R(T^m)+\mathcal{N}(T^n)$ is  closed.
Let  $k\in \N$ such that $k \geq n$ then   $\beta(T_{[k]})\leq \beta(T_{[n]})< \infty,$  hence $\R(T^m)+\mathcal{N}(T^k)$ is  closed.\\
(ii) Let  $n \in \N$ and $m\in \N^*.$  It is easily seen that   $\mathcal{N}(S_{n,m})\cong \mathcal{N}(T^{m+n}).$   We conclude from    Proposition \ref{prop3}  that
\begin{align*}
\alpha(T)<\infty \text{  and }  \beta(T)< \infty &\iff \alpha(T^{m+n})<\infty \text{  and } \beta(T_{[n]})< \infty\\
&\iff  \alpha(S_{n,m})<\infty \text{  and }  \beta(S_{n,m})< \infty.
\end{align*}
Hence $T$ is a Fredholm if and only if $S_{n,m}$ is a Fredholm.
\end{proof}
The next theorem shows that the condition ``$\R(T^{n})$ is closed" cited in  Corollary \ref{cor0} can be omitted.
\begin{thm}\label{thm0}
Let $T \in L(X).$
  If  there exist  $n,m \in \N$  such that $\alpha(T_{[n]})<\infty$ and $\beta(T_{[m]})<\infty,$ then $\R(T^{m_{T}})$ is closed and  $T_{[m_{T}]}$ is a Fredholm operator.
\end{thm}
\begin{proof}
Let $n,m \in \N$  such that $\alpha(T_{[n]})<\infty$ and $\beta(T_{[m]})<\infty.$ From Corollary \ref{cor00} we have $\mbox{max}\{\alpha(T_{[m_{T}]}), \beta(T_{[m_{T}]})\}<\infty.$  The case of  $m_{T}=0$ is clear. Suppose that $m_{T} \geq 1,$  from  \cite[Lemma 3.1, Lemma 3.2]{kaashoek} we have $\frac{\mathcal{N}(T^{m_{T}+k})}{\mathcal{N}(T^{m_{T}})}\cong \mathcal{N}(T^k)\cap \R(T^{m_{T}})$  and $\frac{\R(T^{m_{T}})}{\R(T^{m_{T}+k})}\cong \frac{X}{\R(T^k)+\mathcal{N}(T^{m_{T}})},$ $\forall k \in \N.$ Hence $\dim(\mathcal{N}(T^{m_{T}})\cap \R(T^{m_{T}}))=\displaystyle\sum_{k=m_{T}}^{2m_{T}-1}\alpha(T_{[k]})$ and $\mbox{codim}(\R(T^{m_{T}})+\mathcal{N}(T^{m_{T}}))=\displaystyle\sum_{k=m_{T}}^{2m_{T}-1}\beta(T_{[k]}).$ Since   $(\alpha(T_{[n]}))_{n}$ and $(\beta(T_{[n]}))_{n}$ are decreasing sequences, then $\dim(\mathcal{N}(T^{m_{T}})\cap \R(T^{m_{T}}))< \infty$ and $\mbox{codim}(\R(T^{m_{T}})+\mathcal{N}(T^{m_{T}}))< \infty.$ So   $\mathcal{N}(T^{m_{T}})\cap \R(T^{m_{T}})$ is closed and by Lemma $\ref{lem4},$ $\R(T^{m_{T}})+\mathcal{N}(T^{m_{T}})$ is   closed. Since  $\R(T^{m_{T}})$  is a  paracomplete space,  then it follows from Neubauer Lemma \cite[Proposition 2.1.1]{labrousse} that  $\R(T^{m_{T}})$ is closed. Consequently, $T$ is a B-Fredholm operator.
\end{proof}
From the proof of Theorem \ref{thm0}, it follows that $T \in L(X)$ is B-Fredholm if and only if  $\dim(\mathcal{N}(T^{n})\cap \R(T^{n}))< \infty$ and $\mbox{codim}(\mathcal{N}(T^{n})+\R(T^{n}))< \infty$ for some integer $n\in \N.$
\begin{dfn} We say that $T \in L(X)$ is quasi upper semi-B-Fredholm (resp., quasi  lower semi-B-Fredholm) operator  if there exists $(M,N)\in Red(T)$ such that    $T_{M}$ is upper semi-Fredholm (resp., lower semi-Fredholm)   and $T_{N}$ is  nilpotent.  If $T$ is  quasi upper semi-B-Fredholm  or (resp., and)  quasi lower semi-B-Fredholm then $T$ is called quasi semi-B-Fredholm (resp., quasi  B-Fredholm).
\end{dfn}
Every nilpotent operator and every semi-Fredholm operator are   quasi semi-B-Fredholm. Moreover, every quasi semi-B-Fredholm operator  is   pseudo semi-B-Fredholm. For more details about pseudo semi-B-Fredholm operators, we refer the reader to   \cite{tajmouati}.
\begin{prop}\label{propquasi}  If $T \in L(X)$ is quasi upper semi-B-Fredholm (resp., quasi  lower semi-B-Fredholm) operator   then  $T$  is an upper semi-B-Fredholm (resp.,  a lower semi-B-Fredholm).
\end{prop}
\begin{proof}  Let $(M, N) \in Red(T)$ such that $T_{M}$ is semi-Fredholm and $T_{N}$ is nilpotent.  From Corollary \ref{corkato} below     there exists $(A, B) \in Red(T_{M})$ such that $T_{A}$ is semi-Fredholm and semi-regular  and $T_{B}$ is nilpotent. So $(A, B\oplus N) \in Red(T)$ and $T_{B\oplus N}$ is nilpotent of degree $d.$ Hence $\R(T^{d})=\R(T^{d}_{A})$ and  $\mathcal{N}(T_{A})=\mathcal{N}(T_{[d]})$ and $T(A)\oplus (B\oplus N)=\mathcal{N}(T^{d})+\R(T).$ Therefore $\alpha(T_{A})=\alpha(T_{[d]})$ and $\beta(T_{A})=\beta(T_{[d]}).$  So  $\R(T^{d})$ is closed  and $T_{[d]}$ is semi-Fredholm. Thus $T$ is   semi-B-Fredholm.
\end{proof}
The converse of Proposition \ref{propquasi} is true if $X$ is a Hilbert space, see \cite[Theorem 2.6]{berkani-sarih}.
\vspace{0.2cm}
\par \noindent{\bf Open question:} Does there exist   a  semi-B-Fredholm  operator $T$  acting on a Banach space $X$  which is   not quasi semi-B-Fredholm?

\vspace{0.2cm}
We recall that  $T \in L(X)$ is Drazin invertible if $T$ is a  direct sum of an invertible operator and  a nilpotent operator, and $T$  is meromorphic if $T-\lambda I$ is Drazin invertible for every $\lambda \in \C\setminus\{0\}.$  Note that every nilpotent operator is meromorphic.
\begin{dfn}\cite{rwassa} $T \in L(X)$ is generalized Drazin-meromorphic semi-Fredholm  operator if $T=T_{1}\oplus T_{2};$ where $T_{1}$ is semi-Fredholm and $T_{2}$ is meromorphic.
\end{dfn}
Hereafter,  we denote  by  $\sigma_{se}(T),$  $\sigma_{usf}(T),$  $\sigma_{lsf}(T),$  $\sigma_{e}(T),$  $\sigma_{sf}(T),$ $\sigma_{sbf}(T)$ and  $\sigma_{d}(T)$  respectively, the semi-regular spectrum, the upper  semi-Fredholm spectrum, the lower semi-Fredholm spectrum, the essential  spectrum, the semi-Fredholm spectrum, the  semi-B-Fredholm spectrum and the  Drazin invertible  spectrum of  $T \in L(X).$ We also denote  by $A^{C}$ the complementary of a given complex subset $A.$
\begin{lem}\label{thmindex}
Let $T \in L(X)$ such that  there exist $(M,N),(M^{'},N^{'})\in Red(T)$ with $T_{M}$ and $T_{M^{'}}$ are semi-Fredholm, $T_{N}$ and $T_{N^{'}}$ are meromorphic. Then  $\mbox{ind}(T_{M})=\mbox{ind}(T_{M^{'}}).$
\end{lem}
\begin{proof}
Since $T_{M}$ and  $T_{M^{'}}$ are  semi-Fredholm operators   then from punctured neighborhood theorem for semi-Fredholm operators, there  exists $\epsilon >0$ such that $B(0, \epsilon) \subset \sigma_{sf}(T_{M})^C\cap \sigma_{sf}(T_{M^{'}})^C,$  $\mbox{ind}(T_{M}- \lambda I) =\mbox{ind}(T_{M})$ and   $\mbox{ind}(T_{M^{'}}- \lambda I) =\mbox{ind}(T_{M^{'}})$ for every $\lambda \in B(0, \epsilon).$   As
$T_{N}$ and  $T_{N^{'}}$ are  meromorphic then $B_{0}:=B(0, \epsilon)\setminus\{0\}\subset  \sigma_{sf}(T_{M})^C\cap \sigma_{sf}(T_{M^{'}})^C\cap  \sigma_{d}(T_{N})^C\cap \sigma_{d}(T_{N^{'}})^C\subset\sigma_{sbf}(T)^C.$  Let $\lambda \in B_{0},$  then $T-\lambda I$ is semi-B-Fredholm and      $\mbox{ind}(T- \lambda I)=\mbox{ind}(T_{M}- \lambda I)+\mbox{ind}(T_{N} - \lambda I)=\mbox{ind}(T_{M^{'}}- \lambda I)+\mbox{ind}(T_{N^{'}} - \lambda I).$ Thus  $\mbox{ind}(T_{M})=\mbox{ind}(T_{M^{'}}),$ since the index of a Drazin invertible operator  is equal to zero.
\end{proof}
The previous lemma gives meaning to the following definition.
\begin{dfn} Let $T=T_{1}\oplus T_{2}$  be a generalized Drazin-meromorphic semi-Fredholm  operator; where $T_{1}$ is semi-Fredholm and $T_{2}$ is meromorphic. Then we define the  index of $T$ as the usual   index of the semi-Fredholm  operator  $T_{1}.$ In particular, if $T$ is semi-Fredholm then we find its  usual index.
\end{dfn}
In the following  corollary,  we extend  a particular  case of  \cite[Theorem 2.4]{berkani-castro}    to the general case of  Banach spaces.
\begin{cor}\label{corrr} $T \in L(X)$ is quasi B-Fredholm if and only if $T$ is B-Fredholm. Moreover, if this  is the case then  the index of $T$ as a quasi B-Fredholm coincides with its usual index as a B-Fredholm.
\end{cor}
\begin{proof}  From     \cite[Theorem 7]{muller},  we have  $T$ is B-Fredholm  if and only if  $T = T_1 \oplus T_2;$ where   $T_1$ is  Fredholm  and $T_2$ is nilpotent. So if $T$ is B-Fredholm then  $T$ is quasi B-Fredholm. Suppose that $T$ is quasi B-Fredholm, then    there exist  $(M, N),$ $(M^{'}, N^{'})\in Red(T)$  such that $T_{M}$ is an upper  semi-Fredholm, $T_{M^{'}}$ is a lower  semi-Fredholm, $T_{N}$ and  $T_{N^{'}}$ are nilpotent.   From Lemma \ref{thmindex} we have $\mbox{ind}(T_{M})=\mbox{ind}(T_{M^{'}}),$ and so  $(\alpha(T_{M})+\beta(T_{M^{'}}))-\alpha(T_{M^{'}})=\beta(T_{M})\geq 0.$ Thus   $T_{M}$ and $T_{M^{'}}$  are   Fredholm  operators. Again by \cite[Theorem 7]{muller} we deduce that $T$ is B-Fredholm. Suppose that $T = T_1 \oplus T_2$ is  a B-Fredholm operator. Since $T_1$ is Fredholm, then    \cite[Corollary 4.7]{grabiner} implies that    there exists $\epsilon >0$ such that $B_{0}:=B(0, \epsilon)\setminus \{0\} \subset  (\sigma_{e}(T))^C,$  $\mbox{ind}(T_1 - \lambda I) =\mbox{ind}(T_1)$ and  $\mbox{ind}(T - \lambda I) =\mbox{ind}(T_{[m_{T}]})$ for all $\lambda \in B_{0}.$  Let $\lambda \in B_{0},$ as  $T_2$ is nilpotent then  $T-\lambda I = (T_1-\lambda I) \oplus (T_2 - \lambda I)$ is a Fredholm operator and $\mbox{ind}(T - \lambda I)=\mbox{ind}(T_1 - \lambda I).$ Consequently, $\mbox{ind} (T) = \mbox{ind}(T_1)=\mbox{ind} (T_{[m_{T}]}).$   So  the index of $T$ as a quasi B-Fredholm coincides with its usual index as a B-Fredholm.
\end{proof}
\begin{dfn}\cite{labrousse,kato} $T \in L(X)$  is said to be  of  Kato-type  of degree $d$  if there exists    $(M,N)\in Red(T)$ such that    $T_{M}$ is  semi-regular    and $T_{N}$ is  nilpotent of degree $d.$
\end{dfn}
\begin{rema}\label{remunique} The degree $d$ of a Kato-type operator $T \in L(X)$ is  unique (see \cite[Remark p.205]{labrousse}).
\end{rema}
It is easy to see  that if $T\in L(X)$ is of Kato-type of degree $d$ then $T$    is quasi-Fredholm of degree $\mbox{dis}(T),$   $\alpha(T_{M})=\alpha(T_{[n]})$ and  $\beta(T_{M})=\beta(T_{[n]}),$ for every $n\geq d$ and  $(M,N)\in Red(T)$ such that $T_{M}$ is semi-regular and $T_{N}$ is nilpotent. Recall that  the reduced minimal modulus $\gamma(T)$ of an operator $T$ is defined by $\gamma(T):=\underset{x\notin \mathcal{N}(T)}{\mbox{inf}}\,\frac{\|Tx\|}{d(x,\mathcal{N}(T))}.$  The following proposition gives some characterizations  of      quasi semi-B-Fredholm operators.
\begin{prop}   Let $T \in L(X).$ The following statements are equivalents.\\
(i)  $T$ is  quasi semi-B-Fredholm {\rm[}resp.,  quasi upper semi-B-Fredholm, quasi lower  semi-B-Fredholm, quasi  B-Fredholm{\rm]};\\
(ii) $T$  is  of Kato-type of degree $d$ and   $\mbox{min}\,\{\alpha(T_{[d]}), \beta(T_{[d]})\}<\infty$  {\rm[}resp.,   $T$  is  of Kato-type of degree $d$ and  $\alpha(T_{[d]})<\infty,$   $T$  is  of  Kato-type  of degree $d$  and  $\beta(T_{[d]})<\infty,$     $\mbox{max}\,\{\alpha(T_{[n]}), \beta(T_{[n]})\}<\infty$ for an integer $n${\rm]};\\
(iii)  $T$  is of Kato-type   and  $0 \notin \mbox{acc}\,\sigma_{sf}(T)$ {\rm[}resp.,  $T$  is of  Kato-type and $0 \notin \mbox{acc}\,\sigma_{usf}(T),$ $T$  is of Kato-type and $0 \notin \mbox{acc}\,\sigma_{lsf}(T),$  $T$  is of Kato-type and $0 \notin \mbox{acc}\,\sigma_{e}(T)${\rm]}.
\end{prop}
\begin{proof}   Remark firstly   from   Theorem \ref{thm0} and Corollary \ref{corrr}  that  if $\mbox{max}\,\{\alpha(T_{[n]}), \beta(T_{[n]})\}<\infty$ for an integer $n,$ then $T$  is  quasi B-Fredholm.\\
$(i) \Longleftrightarrow (ii)$ Suppose that $T$ is quasi semi-B-Fredholm, then  there exists  $(A,B)\in Red(T)$ such that $T_{A}$ is semi-Fredholm and $T_{B}$ is nilpotent.   According to  Corollary \ref{corkato}, there exists  a pair  $(M, N)\in Red(T)$ such that    $T_{M}$ is semi-Fredholm which is semi-regular and $T_{N}$ is nilpotent of degree $d.$  Thus $T$ is  of Kato-type  of degree $d$    and $\mbox{min}\,\{\alpha(T_{[d]}), \beta(T_{[d]})\}=\mbox{min}\,\{\alpha(T_{M}), \beta(T_{M})\}<\infty.$ The converse is obvious.  The other    equivalence cases  go similarly.\\
$(iii) \Longleftrightarrow (i)$ Assume that $T$ is of Kato-type  and  $0 \notin \mbox{acc}\,\sigma_{sf}(T);$ where $\mbox{acc}A$ is the accumulation points of a given complex subset $A.$   Then  there exists  $(M,N)\in Red(T)$ such that    $T_{M}$ is  semi-regular    and $T_{N}$ is  nilpotent of degree $d.$  From the  punctured neighborhood theorem for semi-regular  operators,  $B_{0}:=B(0,\epsilon)\setminus\{0\} \subset (\sigma_{se}(T)\cup\sigma_{sf}(T))^{C};$ for some   $\epsilon\leq\gamma(T_{M}).$  On the other hand, it is easily seen that  $\alpha(T_{M})=\alpha(T_{[d]})=\alpha(T-\lambda I)$ and  $\beta(T_{M})=\beta(T_{[d]})=\beta(T-\lambda I),$ thus $\mbox{min}\,\{\alpha(T_{M}), \beta(T_{M})\}=\mbox{min}\,\{\alpha(T-\lambda I), \beta(T-\lambda I)\}<\infty$  for every $\lambda \in B_{0}.$  Thus $T_{M}$ is semi-Fredholm and  so $T$ is  quasi semi-B-Fredholm. The converse is an immediate consequence of the  punctured neighborhood theorem for semi-Fredholm operators.  The other    equivalence cases  go similarly.
\end{proof}
In the sequel we  denote by $\sigma_{qsbf}(T),$ $\sigma_{qusbf}(T),$ $\sigma_{qlsbf}(T)$  and  $\sigma_{qbf}(T)$   respectively,  the quasi   semi-B-Fredholm spectrum, the quasi upper  semi-B-Fredholm spectrum, the quasi lower  semi-B-Fredholm spectrum and the quasi B-Fredholm spectrum of $T \in L(X).$ The second point  of the next corollary is a consequence of the classical  Heine-Borel theorem.
\begin{cor} For every $T \in L(X)$ we have\\
(i) $\sigma_{qsbf}(T),$ $\sigma_{qusbf}(T),$ $\sigma_{qlsbf}(T)$  and  $\sigma_{qbf}(T)$  are  compact.\\
(ii) If $\Omega$ is a component of $(\sigma_{uqsbf}(T))^C$ or  $(\sigma_{qlsbf}(T))^C,$ then the index $\mbox{ind}(T -\lambda I)$ is constant as $\lambda$ ranges over $\Omega.$
\end{cor}
In \cite[Theorem 3.2.2]{labrousse}, Labrousse  showed that the degree  of a Kato-type   operator  $T$ acting on a Hilbert space  is exactly  its  degree of stable iteration   $\mbox{dis}(T).$  In the next theorem we extend this result to  the general case of Banach space Kato-type operators.
\begin{thm}\label{thmunique}  The degree of a Kato-type operator  $T \in L(X)$ is equal to  $\mbox{dis}(T).$
\end{thm}
\begin{proof} Let $(M, N) \in  Red(T)$ such that  $T_{M}$ is semi-regular and  $T_{N}$ is nilpotent of degree $d.$      Let   $m\geq d,$ since $T_{M}$ is semi-regular then  $\R(T)+\mathcal{N}(T^{m})=\R(T_{M})+\mathcal{N}(T_{M}^{m})+\R(T_{N})+\mathcal{N}(T_{N}^{m})= \R(T_{M})\oplus N=\R(T_{M})+\mathcal{N}(T_{M}^{d})+\R(T_{N})+\mathcal{N}(T_{N}^{d})=\R(T)+\mathcal{N}(T^{d}).$ Thus $d\geq \mbox{dis}(T).$ On the other hand,  $T_{N}$ is  a  quasi B-Fredholm operator   then there exists according to \cite[Theorem 7]{muller}  a pair    $(A, B) \in  Red(T_{N})$ such that  $T_{A}$ is  semi-regular and $T_{B}^{dis(T_{N})}=0.$ The same reasoning as above shows that  $T_{B}$ is nilpotent of degree $\mbox{dis}(T_{N}).$  It is easily seen that $(M\oplus A, B)\in Red(T)$ and  $T_{M\oplus A}$  is semi-regular.   Since  the degree of a Kato-type operator is unique then $d=\mbox{dis}(T_{N})\leq \mbox{dis}(T)=\mbox{max}\,\{\mbox{dis}(T_{M}),\mbox{dis}(T_{N})\}.$ Hence $d=\mbox{dis}(T).$
\end{proof}
\begin{cor} The degree of a  nilpotent operator $T \in L(X)$ is exactly $\mbox{dis}(T).$
\end{cor}
 Let $M$ be a subset of  $X$ and  $N$ a subset of $X^{*}.$ The annihilator of $M$ and the pre-annihilator of  $N$  are  the closed subspaces  defined respectively, by
$$M^{\perp} := \{f \in X^{*} : f (x) = 0 \text{ for every }  x \in M\},$$
and
$${}^{\perp}N := \{x \in X : f (x) = 0 \text{ for every } f \in N\}.$$
\begin{prop}\label{proppp} If $T \in L(X)$ is  of Kato-type of degree $d$ then $T^{*}$ is   of  Kato-type with  the same degree $d.$   Moreover, $\R(T^{*})+\mathcal{N}(T^{*d})$  is closed in the weak-*-topology  $\sigma(X^{*}, X)$ on $X^{*}.$
\end{prop}
\begin{proof}  It is well known that $T^{*}$ is   of Kato-type of degree noted  $d^{'}.$ From  Theorem \ref{thmunique} we have $d=\mbox{dis}(T)$ and  $d^{'}=\mbox{dis}(T^{*}).$ So it suffices to show   that $\mbox{dis}(T)=\mbox{dis}(T^{*}).$   Since $\R(T^{p})$  is closed and    $\R(T^{*q})$ is  closed  for every $p\geq d$ and  $q\geq d^{'},$   then $\R(T^{n})$ is also  closed for every $n\geq b:=\mbox{min}\{d,d^{'}\}.$   This  implies from    \cite[Lemma 10]{mbekhta} that    $\R(T)+\mathcal{N}(T^{n})$ is closed  for every $n\geq b.$  Hence  $(\R(T)+\mathcal{N}(T^{n}))^{\perp}=\mathcal{N}(T^{*})\cap\R(T^{*n})$  and    $\R(T)+\mathcal{N}(T^{n})={}^{\perp}(\mathcal{N}(T^{*})\cap\R(T^{*n}))$ for every $n\geq b.$ Consequently $d=d^{'}.$      On the other hand, it is easily seen that if $T=T_{M}\oplus T_{N}$ is a decomposition of Kato-type  of $T$ then $T^{*}=T^{*}_{N^{\perp}}\oplus T^{*}_{M^{\perp}}$ is a decomposition of Kato-type of $T^{*}.$      From    the proof of  Theorem \ref{thmunique}  we obtain that   $\R(T^{*}_{N^{\perp}})\oplus M^{\perp}=\R(T^{*})+\mathcal{N}(T^{*d}).$ It is easy to get  that  $P_{M^{\perp}}=(P_{N})^{*}$ and $P_{N^{\perp}}=(P_{M})^{*},$ so    $(TP_{M}+ P_{N})^{*}=T^{*}P_{N^{\perp}}+P_{M^{\perp}}.$   As  $\R(TP_{M}+ P_{N})=\R(T_{M})\oplus N$ is  closed  then    \cite[Theorem A.16]{mullerbook} entails that $\R((TP_{M}+ P_{N})^{*})=\R(T^{*})+\mathcal{N}(T^{*d})$     is $w^{*}$-closed.
\end{proof}
From the  proof of Proposition \ref{proppp}, it follows that if $T \in L(X)$ is quasi-Fredholm then $T^{*}$ is also  quasi-Fredholm and $\mbox{dis}(T)=\mbox{dis}(T^{*}).$
\section{A relation between $m_{T},$ $v(T)$ and $dis(T)$}
Let $T \in L(X).$ As mentioned   in the introduction,  the degree of stable iteration $d:=\mbox{dis}(T)$ of $T$ is defined as  $d=\mbox{inf}\,\Delta(T);$   where    $$\Delta(T):=\{m\in\N \,:\,  \mathcal{N}(T_{[m]})= \mathcal{N}(T_{[r]}),\,\forall r \in \N  \,  \, r\geq m \},$$
with $\mbox{inf}\,\emptyset=\infty.$
Let  $r \in \N,$  then  $r\geq d$ if and only if  $\R(T)+ \mathcal{N}(T^{m})=\R(T)+ \mathcal{N}(T^{r}),\,\forall m \in \N  \,\text{  such that}  \, m\geq r,$ see \cite[Theorem 3.1]{grabiner}.\\
 If $T$ is  an upper  semi-B-Fredholm operator then from \cite[Proposition 2.5]{berkani-sarih}, $T$ is quasi-Fredholm of degree $d\geq m_{T}.$ From  the punctured neighborhood theorem for semi-B-Fredholm operators \cite[Theorem 4.7]{grabiner},   there exists $\epsilon:=\gamma(T_{\R(T^{d})})>0$ such that $\lambda \longrightarrow \alpha(T-\lambda I)$ is constant on  $B(0, \epsilon)\setminus\{0\},$  $\alpha(T-\lambda I)\leq \alpha(T_{[m_T]})$ and $\mbox{ind}(T_{[m_T]})=\mbox{ind}(T-\lambda I)$ for every $\lambda\in B(0, \epsilon).$    Analogously,  if   $T$ is a lower   semi-B-Fredholm operator then   $\lambda \longrightarrow  \beta(T-\lambda I)$ is constant on  $B(0, \epsilon)\setminus\{0\},$  $\beta(T-\lambda I)\leq \beta(T_{[m_T]})$ and $\mbox{ind}(T_{[m_T]})=\mbox{ind}(T-\lambda I)$ for every $\lambda\in B(0, \epsilon).$
\begin{dfn}\label{dfn10}  Let $T\in L(X)$ be a semi-B-Fredholm operator.   The jump of $T$ is defined by
$$\mbox{jump}(T)=\left\lbrace
\begin{array}{ll}
\alpha(T_{[m_T]})-\alpha(T-\lambda I) & \mbox{if $T$ is  an upper   semi-B-Fredholm operator}\\
\beta(T_{[m_T]})-\beta(T-\lambda I) & \mbox{if $T$ is  a lower   semi-B-Fredholm operator}
\end{array}
\right.$$
where  $\lambda\in B(0, \epsilon)\setminus\{0\}$ be  arbitrary and  $\epsilon=\gamma(T_{\R(T^{d})})>0.$
\end{dfn}
 The definition of the $\mbox{jump}(T)$ of a semi-B-Fredholm operator  is well defined since  $\mbox{ind}(T_{[m_T]})=\mbox{ind}(T-\lambda I).$  Furthermore, if $T$ is  semi-Fredholm   we find its  jump  given by T. T. West in \cite{west}.
\vspace{0.3cm}
\par The next proposition extends   \cite[Proposition 3]{west}, which establishes that if $T \in L(X)$ is semi-Fredholm then $T$ is semi-regular if and only if $\mbox{jump}(T)=0.$  Note that here $m_{T}=0.$
\begin{prop}\label{prop11}   Let $T\in L(X)$  be a semi-B-Fredholm operator. Then $T_{[m_T]}$ is semi-regular if and only if $\mbox{jump}(T)=0.$
\end{prop}
\begin{proof}  Let $\lambda \in  B(0, \epsilon)\setminus\{0\}$. From \cite[Theorem 4.7]{grabiner}, $\alpha(T-\lambda I)=\mbox{dim}\,\frac{\mathcal{N}(T^{d+1})}{\mathcal{N}(T^{d})}.$ Moreover, by    \cite[Lemma 3.1]{kaashoek}  we have  $\mathcal{N}(T_{[d]})\cong \frac{\mathcal{N}(T^{d+1})}{\mathcal{N}(T^{d})}$ and so $\alpha(T-\lambda I)=\alpha(T_{[d]}).$ Therefore $\mbox{jump}(T)=0$ if and only if $\mathcal{N}(T_{[d]})=\mathcal{N}(T_{[m_T]})$ if and only if $d=m_T.$ Since $T_{[d]}$ is semi-regular  then we get the desired result.
\end{proof}

\begin{thm}\label{thmm} If  $T \in L(X)$  is  a semi-B-Fredholm operator then there exists $(M,N)\in Red(T)$ such that  $\mbox{dim}\,N<\infty,$   $\mbox{jump}(T_{M})=0$ (i.e. $\mbox{dis}(T_{M})=m_{T}$)  and $T_{N}$ is  nilpotent  of degree $d.$ Moreover $\mbox{ind}(T)=\mbox{ind}(T_{M}).$
\end{thm}
Before giving  the proof of this theorem we need  the following   lemmas.
\begin{lem}\label{lemm1} Let   $T \in L(X)$ be   an operator with the degree of  stable  iteration  $0<d<\infty.$  Then for every $y\in \mathcal{N}(T^{d})\setminus(\mathcal{N}(T^{d-1})+\R(T))$  we have $T^{i}y \in \mathcal{N}(T^{d-i})\setminus(\mathcal{N}(T^{d-i-1})+\R(T^{i+1})),$  $i=0, ..., d-1,$ and $\{T^{i}y\}_{i=0}^{d-1}$ are linearly independent modulo $\R(T^{d}).$
\end{lem}
\begin{proof}
Since $0<d<\infty$ then  $\mathcal{N}(T^{d})\not\subset\mathcal{N}(T^{d-1})+\R(T).$ Let  $y\in \mathcal{N}(T^{d})\setminus(\mathcal{N}(T^{d-1})+\R(T))$ and $i\in\{0, ..., d-1\},$ then $T^{i}y \in \mathcal{N}(T^{d-i}).$  Suppose  that $T^{i}y \in \mathcal{N}(T^{d-i-1})+\R(T^{i+1}),$ then  $y\in  T^{-i}(\mathcal{N}(T^{d-i-1})+\R(T^{i+1}))=\mathcal{N}(T^{d-1})+\R(T)$ and this is a contradiction. Thus $T^{i}y \in \mathcal{N}(T^{d-i})\setminus(\mathcal{N}(T^{d-i-1})+\R(T^{i+1})).$\\ Let $(\alpha_{i})_{i=0}^{d-1}\subset  \C$ such that $\displaystyle\sum_{i=0}^{d-1} \alpha_{i} T^{i}y\in \R(T^{d}).$ Then $\alpha_{0}T^{d-1}y=T^{d-1}(\displaystyle\sum_{i=0}^{d-1} \alpha_{i} T^{i}y)\in \R(T^{2d-1}).$ Thus $\alpha_{0}y\in T^{-(d-1)}(\R(T^{2d-1}))=\mathcal{N}(T^{d-1})+\R(T^{d})\subset \mathcal{N}(T^{d-1})+\R(T),$ and this  implies that  $\alpha_{0} = 0.$ If $d\geq2,$    then by   similar arguments we get  $\alpha_{j}=0,$ for  $j\in\{0,\dots, i\};$ where $i=0, ..., d-2.$ Therefore  $\alpha_{i+1}y\in  T^{-(d-1)}(\R(T^{2d-i-2}))=\mathcal\mathcal{N}(T^{d-1})+\R(T^{d-i-1})\subset \mathcal{N}(T^{d-1})+\R(T),$ since $d-i\geq 2.$ Thus $\alpha_{i+1}=0$ and so $y, Ty, ..., T^{d-1}y$ are linearly independent modulo $\R(T^{d}).$
\end{proof}
\begin{lem}\label{lemm2}  Let  $T \in L(X)$ be  a quasi-Fredholm operator of degree $d>0.$  Then for every  $y\in \mathcal{N}(T^{d})\setminus(\mathcal{N}(T^{d-1})+\R(T))$   there exists $f \in X^{*}$  satisfies the following\\
(a)    $T^{i*}f(T^{d-j-1}y)=\delta_{ij},$ for every $0\leq i,j\leq d-1.$\\
(b) $T^{i*}f \in \mathcal{N}(T^{(d-i)*})\setminus(\mathcal{N}(T^{(d-i-1)*})+\R(T^{(i+1)*})),$ for  every $0\leq i \leq d-1.$\\
(c)  The cascade  $\{T^{i*}f\}_{i=0}^{d-1}$  is  linearly independent modulo $\R(T^{d*}).$\\
Such cascade  $\{T^{i*}f\}_{i=0}^{d-1}$  called  adjoint cascade of $\{T^{i}y\}_{i=0}^{d-1}.$
\end{lem}
\begin{proof}   We  denote by $Y=<y,\dots,T^{d-1}y>$ the subspace spanned by the cascade $\{T^{i}y\}_{i=0}^{d-1}.$ Then Lemma \ref{lemm1} implies that  $Y \cap \R(T^{d})=\{0\}.$  As  $\R(T^{d})$ is closed,  according to   Hahn-Banach theorem there exists $f \in \R(T^{d})^{\perp}=\mathcal{N}(T^{d*})$ such that $f(T^{d-i}y)=\delta_{i1},$ for every $0\leq i\leq d-1.$ It is not difficult  to see that  $T^{*i}f(T^{d-j-1}y)=\delta_{ij},$ for every $0\leq i,j\leq d-1.$   Moreover, $f \in \mathcal{N}(T^{d*})\setminus(\mathcal{N}(T^{(d-1)*})+\R(T^{*})).$ Indeed, if  $f \in \mathcal{N}(T^{(d-1)*})+\R(T^{*})$ then there exists $(g,h)\in  \mathcal{N}(T^{(d-1)*})\times X^{*}$ such that  $f=g+T^{*}h,$ thus $f(T^{d-1}y)=g(T^{d-1}y)+T^{*}h(T^{d-1}y)=0,$ and this is impossible.     On the other hand, since   $0<\mbox{dis}(T^{*})=d<\infty$ (see Proposition \ref{proppp})  we conclude from Lemma  \ref{lemm1} that the  linear form  $f$ satisfies the points $(b)$ and $(c).$
\end{proof}
Let $(g,x) \in X^{*}\times X$ be non-zero. We denote by $g\otimes x$ the rank-one operator defined by $(g\otimes x)z=g(z)x$ for all $z\in X.$  Note that every rank-one operator in $L(X)$ can be written in this form.
\begin{lem}\label{lemm3}    Let  $T \in L(X)$ be  a quasi-Fredholm operator of degree $d>0,$     $y\in \mathcal{N}(T^{d})\setminus(\mathcal{N}(T^{d-1})+\R(T))$ and $\{T^{i*}f\}_{i=0}^{d-1}$  be an  adjoint cascade of  $\{T^{i}y\}_{i=0}^{d-1}.$   Then\\
(a)  $P:=\displaystyle\sum_{i=0}^{d-1} T^{i*}f\otimes T^{d-i-1}y$ is a finite-rank  projection onto $Y$  which commutes with $T;$ where $Y:=<y,\dots,T^{d-1}y>.$\\
(b)  $\mathcal{N}(P)=\{x\in X : \{x, \dots,T^{d-1}x\} \subset \mathcal{N}(f)\}$ is a closed $T$-invariant subspace of $X.$\\
(c)   $T_{Y}$ is nilpotent of degree $d$ and $\mbox{jump}(T_{Y})=1.$
\end{lem}
\begin{proof} (a) Let  $P_{i}=T^{i*}f\otimes T^{d-i-1}y$ be  the rank-one    operator; it is easily seen that  $P_{i} \in L(X),$     $\R(P_{i})=<T^{d-i-1}y>$ and $P_{i}P_{j}=P_{j}P_{i}=\delta_{ij}P_{i},$ for every  $0\leq i,j\leq d-1.$  Thus $P=\displaystyle\sum_{i=0}^{d-1}P_{i} \in L(X)$  is a projection and $\R(P)=Y.$  Moreover, $TP_{i}=T(T^{i*}f\otimes T^{d-i-1}y)=T^{i*}f\otimes T^{d-i}y$ and  $P_{i}T=T^{(i+1)*}f\otimes T^{d-i-1}y,$ $i=0,\dots,d-1.$ Since $y \in \mathcal{N}(T^{d})$ and $f \in \mathcal{N}(T^{d*})$ then  $TP_{0}=P_{d-1}T=0$ and hence $TP=PT.$\\
(b)  It is easily seen that $\mathcal{N}(P)=\{x\in X : \{x,\dots,T^{d-1}x\} \subset \mathcal{N}(f)\}.$ Moreover, it is clear that   $\mathcal{N}(P)$ is closed.  Let $x \in \mathcal{N}(P),$  as $f \in \mathcal{N}(T^{d*})$ then  $\{Tx,\dots,T^{d}x\} \subset \mathcal{N}(f)$ and consequently    $T(\mathcal{N}(P))\subset \mathcal{N}(P).$\\
(c) We have   $\R(T_{Y}^{d-1})=\mathcal{N}(T_{Y})=<T^{d-1}y>$ and   $\R(T_{Y}^{d})=\{0\}.$  So $T_{Y}$ is nilpotent of degree $d$ with      $\mbox{jump}(T_{Y})=1.$
\end{proof}
 Now  we give the proof of Theorem \ref{thmm}.
\begin{proof}[Proof of Theorem \ref{thmm}] Since $T$  is  semi-B-Fredholm then $T$ is quasi-Fredholm of degree $d.$ We distinguish two cases:\\ Case 1:   $d=m_{T}.$   In this case  if $d=0$ then   the   theorem  is  true   with   $M=X$ and $N=\{0\}.$     If $d >0$ then  let      $f$ be the linear form defined in Lemma \ref{lemm2} and let  $P$ be the projection defined in Lemma \ref{lemm3}.  Hence  the proof is complete  by taking $M=\mathcal{N}(P)=\displaystyle\bigcap_{i=0}^{d-1}\,\mathcal{N}(T^{i*}f)$ and $N=\R(P).$  Indeed,   $m_{T}=max\{m_{T_{M}},m_{T_{N}}\}=m_{T_{M}}$ and $d=max\{\mbox{dis}(T_{M}),\mbox{dis}(T_{N})\}\geq \mbox{dis}(T_{M})\geq m_{T}.$ Thus $d=\mbox{dis}(T_{M})=m_{T}$ and $\mbox{jump}(T_{M})=0.$\\
Case 2:   $d>m_{T}.$ Let  $P_{1}\in L(X)$ be the projection defined in Lemma \ref{lemm3}. If we take $M_{1}=\mathcal{N}(P_{1})$ and $N_{1}=\R(P_{1})=<y,\dots, T^{d-1}y>;$ where $y\in \mathcal{N}(T^{d})\setminus(\mathcal{N}(T^{d-1})+\R(T)),$  then  $(M_{1},N_{1})\in Red(T),$ $\mbox{dim}\,N_{1}=d<\infty$ and $T_{N_{1}}$ is a nilpotent operator of degree $d.$ Moreover $\R(T^{i}_{N_{1}})=<T^{i}y,...,T^{d-1}y>$ and $\mathcal{N}(T^{i}_{N_{1}})=<T^{d-i}y,...,T^{d-1}y>,$ $i=0,..., d-1;$ where  $\mathcal{N}(T^{0}_{N_{1}})=\{0\}.$
Thus $\alpha(T_{[m_{T}]})=\alpha((T_{M_{1}})_{[m_{T}]})+1$ and $\beta(T_{[m_{T}]})=\beta((T_{M_{1}})_{[m_{T}]})+1.$
So $\mbox{jump}(T_{M_{1}})= \mbox{jump}(T)- 1.$ In the sequel we denote  by $k:=\mbox{jump}(T)>0.$ Continuing this process   $k$ time,   we get from Lemma \ref{lemm3} a sequence  of  projections $(P_{i})_{i=1}^{k}$ with $P_{i+1}\in L(M_{i}),$ $i=0,\dots, k-1;$  where $(M_{0}, N_{0})=(X, \{0\})$ and   $(M_{i}, N_{i})=(\mathcal{N}(P_{i}),\R(P_{i})),$ $i=1,\dots, k.$      Then   $(M_{i+1}, N_{i+1})\in Red(T_{M_{i}}),$    $d_{i}:=\mbox{dim}\,N_{i+1}=\mbox{dis}(T_{M_{i}})$    and $T_{N_{i+1}}$ is nilpotent of degree $d_{i},$ $i=0,\dots ,k-1,$  $\mbox{jump}(T_{M_{i}})=\mbox{jump}(T)- i,$ $i=0,\dots ,k.$  Let $M:=M_{k}$ and $N:=N_{1}\oplus ... \oplus N_{k}.$  Then $(M, N)\in Red(T),$  $\alpha(T_{[d]})=\alpha((T_{M})_{[m_{T}]}),$  $\beta(T_{[d]})=\beta((T_{M})_{[m_{T}]}),$  $\mbox{jump}(T_{M})=0,$  $\mbox{jump}(T_{N})=k,$ $d+(k-1)m_{T}\leq\mbox{dim}\,N=\displaystyle\sum_{i=0}^{k-1}d_{i}\leq kd$   and $T_{N}$ is nilpotent of degree $d.$ Thus $\mbox{ind}(T)=\mbox{ind}(T_{M}).$
\end{proof}
From Theorem \ref{thmm}, we  immediately obtain the following corollary which is  the main result of \cite{west} (see also \cite[Proposition 2.5]{west1}). And it is also a special case of Kato's decomposition theorem \cite[Theorem 4]{kato}.
\begin{cor}\label{corkato} If  $T \in L(X)$  is a semi-Fredholm operator then there exists $(M,N)\in Red(T)$ such that $\mbox{dim}\,N<\infty,$ $T_{M}$ is semi-regular operator and $T_{N}$ is  nilpotent of degree $\mbox{dis}(T).$
\end{cor}
The next theorem proves  that  the degree of stable iteration $d=\mbox{dis}(T)$ of  a  semi-Fredholm operator $T \in L(X)$   is equal to
$$d=\left\lbrace
\begin{array}{ll}
0 & \mbox{if $T$ is semi-regular}\\
v(T) &  \mbox{else}
\end{array}
\right.$$
 where $v(T):=\mbox{inf}\{n \in \N: \mathcal{N}(T)\not\subset \R(T^{n})\}$  is the degree of the semi-regularity of $T.$ So  if $d>0$ then $\mathcal{N}(T)\subset \R(T^{n}) \text{ and } \mathcal{N}(T)\not\subset \R(T^{d})$ for every $n<d.$  Note that if $T$ is semi-regular then $v(T)=\infty$ and $\mbox{dis}(T)=0.$
\begin{thm}\label{vd} If $T \in L(X)$ is  semi-Fredholm which is not   semi-regular  then $\mbox{dis}(T)=v(T).$
\end{thm}
\begin{proof}    It  is an immediate consequence of \cite[Theorem 4]{kato} (see also \cite[p.609]{west})    and    Theorem \ref{thmunique}.  (We also can  use  \cite[Theorem 4]{kato}, Corollary  \ref{corkato} and Remark \ref{remunique}).
\end{proof}
\begin{cor} If $T \in L(X)$ is semi-Fredholm and $d=\mbox{dis}(T)$ then
$$\alpha(T_{[n]})=\left\lbrace
\begin{array}{ll}
\alpha(T_{[d]}) &  \mbox{if n $\geq$ d}\\
\alpha(T)  &  \mbox{else}
\end{array}
\right.$$
and $$\beta(T_{[n]})=\left\lbrace
\begin{array}{ll}
\beta(T_{[d]}) &  \mbox{if n $\geq$ d}\\
\beta(T) & \mbox{else}
\end{array}
\right.$$
  In particular,
$$\alpha(T^{n})=\left\lbrace
\begin{array}{ll}
n\alpha(T) &  \mbox{if n $\leq$ d}\\
d\alpha(T)+(n-d)\alpha(T_{[d]})  &  \mbox{else}
\end{array}
\right.$$
and $$\beta(T^{n})=\left\lbrace
\begin{array}{ll}
n\beta(T) &  \mbox{if n $\leq$ d}\\
d\beta(T)+(n-d)\beta(T_{[d]})  &  \mbox{else}
\end{array}
\right.$$
\end{cor}
\begin{proof}  We will assume that $T$ is not semi-regular,  otherwise the assertion is trivial.   From  Theorem \ref{vd}  we have  $d=v(T).$ So   $\mathcal{N}(T)\subset \R(T^{n})$ and    therefore  \cite[Lemma 511]{kato} and  \cite[Lemma 3.1, Lemma 3.2]{kaashoek} imply that  $\frac{\mathcal{N}(T^{n+1})}{\mathcal{N}(T^{n})}\cong\mathcal{N}(T)$  and $\frac{\R(T^{n})}{\R(T^{n+1})}\cong \frac{X}{\R(T)}$ for every integer $n<d.$ On the other hand, by the definition of $\mbox{dis}(T)$ we have $\alpha(T_{[d]}) =\alpha(T_{[m]})$ and  $\beta(T_{[d]}) =\beta(T_{[m]})$ for every $m\geq d.$ Moreover, it is easily seen that   $\alpha(T^{m})=\displaystyle\sum_{k=0}^{m-1}\alpha(T_{[k]})$ and  $\beta(T^{m})=\displaystyle\sum_{k=0}^{m-1}\beta(T_{[k]})$ for every $m \in \N^{*}.$ So the proof is complete.
\end{proof}
\begin{cor} If $T \in L(X)$ is semi-Fredholm then  $T$ is semi-regular if and only if  $\mathcal{N}(T)\subset \R(T^{d});$ where $d=\mbox{dis}(T).$
\end{cor}

\goodbreak
{\small \noindent Zakariae Aznay,\\  Laboratory (L.A.N.O), Department of Mathematics,\\Faculty of Science, Mohammed I University,\\  Oujda 60000 Morocco.\\
aznay.zakariae@ump.ac.ma\\

{\small \noindent  Abdelmalek Ouahab,\\  Laboratory (L.A.N.O), Department of Mathematics,\\Faculty of Science, Mohammed I University,\\  Oujda 60000 Morocco.\\
ouahab05@yahoo.fr\\

 \noindent Hassan  Zariouh,\newline Department of
Mathematics (CRMEFO),\newline
 \noindent and laboratory (L.A.N.O), Faculty of Science,\newline
  Mohammed I University, Oujda 60000 Morocco.\\
 \noindent h.zariouh@yahoo.fr

\end{document}